\documentclass[12pt]{article}%
\usepackage[intlimits]{amsmath}
\usepackage{amssymb}
\usepackage[T1]{fontenc}
\usepackage[sc]{mathpazo}
\usepackage{color}
\usepackage[colorlinks]{hyperref}
\usepackage{amsfonts}
\usepackage{graphicx}%
\definecolor {refcol}{RGB}{40,0,255}
\hypersetup{colorlinks=true,allcolors=refcol}
\makeatletter
\newfont{\footsc}{cmcsc10 at 8truept}
\newfont{\footbf}{cmbx10 at 8truept}
\newfont{\footrm}{cmr10 at 10truept}
\newtheorem{theorem}{Theorem}

\newtheorem{problem}[theorem]{Problem}
\newtheorem{proposition}[theorem]{Proposition}

\newenvironment{proof}[1][Proof]{\noindent{\textbf {#1}  }}  {\hfill$\Box$\bigskip}
\begin{document}

\title{\textbf{Hoffman's bound for hypergraphs}}
\author{V. Nikiforov\thanks{Department of Mathematical Sciences, University of
Memphis, Memphis TN 38152, USA. Email: \textit{vnikifrv@memphis.edu}}}
\date{}
\maketitle

\begin{abstract}
One of the best-known results in spectral graph theory is the inequality of
Hoffman%
\[
\chi\left(  G\right)  \geq1-\frac{\lambda\left(  G\right)  }{\lambda_{\min
}\left(  G\right)  },
\]
where $\chi\left(  G\right)  $ is the chromatic number of a graph $G$ and
$\lambda\left(  G\right)  ,$ $\lambda_{\min}\left(  G\right)  $ are the
largest and the smallest eigenvalues of its adjacency matrix.

In this note Hoffman's inequality is extended to weighted uniform $r$-graphs
for every even $r$.\medskip



\end{abstract}

\section{Introduction and main results}

Let $G$ be a graph with chromatic number $\chi\left(  G\right)  $, and let
$\lambda\left(  G\right)  $ and $\lambda_{\min}\left(  G\right)  $ be the
largest and the smallest adjacency eigenvalues of $G.$ A famous result of
Hoffman \cite{Hof70} asserts that
\begin{equation}
\chi\left(  G\right)  \geq1-\frac{\lambda\left(  G\right)  }{\lambda_{\min
}\left(  G\right)  }. \label{hin}%
\end{equation}
Inequality (\ref{hin}) is tight, as equality holds for a vast class of graphs
(see \cite{GoRo01}, Lemma 9.6.2), in particular, for regular complete
multipartite graphs.

The main goal of this paper is to extend inequality (\ref{hin}) to weighted
$r$-uniform hypergraphs for every even $r\geq2$.

Let $r\geq2$ and $G$ be an $r$-uniform hypergraph ($r$\emph{-graph} for short)
with vertex set $V\left(  G\right)  $ and edge set $E\left(  G\right)  .$ For
simplicity, we let $V\left(  G\right)  =\left[  n\right]  $, where $\left[
n\right]  :=\left\{  1,\ldots,n\right\}  $.

If a function $w:E\left(  G\right)  \rightarrow\mathbb{R}$ is associated with
$G$, then $G$ is called a \emph{weighted }$r$\emph{-graph} and is denoted by
$G_{w}$. If for an $r$-graph $G$ the function $w$ is not given explicitly, it
is assumed that $w\left(  e\right)  =1$ for each $e\in E\left(  G\right)  $.

Given a weighted $r$-graph $G_{w}$, write $w_{i_{1},\ldots,i_{r}}$ for
$w\left(  \left\{  i_{1},\ldots,i_{r}\right\}  \right)  $ for any $\left\{
i_{1},\ldots,i_{r}\right\}  \in E\left(  G_{w}\right)  $. The \emph{polynomial
form} $P_{G_{w}}$ of $G_{w}$ is a function $P_{G_{w}}:\mathbb{R}%
^{n}\rightarrow\mathbb{R}$ defined for any vector real $n$-vector
$\mathbf{x}:=\left(  x_{1},\ldots,x_{n}\right)  $ as
\[
P_{G_{w}}\left(  \mathbf{x}\right)  :=r!\sum_{\left\{  i_{1},\ldots
,i_{r}\right\}  \in E\left(  G\right)  }w_{i_{1},\ldots,i_{r}}x_{i_{1}}\cdots
x_{i_{r}}%
\]

Note that if $G$ is a $2$-graph and $w\left(  e\right)  =1$ for any $e\in
E\left(  G\right)  ,$ then $P_{G}\left(  \mathbf{x}\right)  $ is the quadratic
form of the adjacency matrix of $G$, and the Rayleigh-Ritz theorem (see
\cite{Hof70}, Theorem 4.2.4) asserts that
\[
\lambda\left(  G\right)  =\max_{x_{1}^{2}+\cdots+x_{n}^{2}=1}\text{ }%
P_{G}\left(  \mathbf{x}\right)  \text{ \ \ and \ \ }\lambda_{\min}\left(
G\right)  =\min_{x_{1}^{2}+\cdots+x_{n}^{2}=1}\text{ }P_{G}\left(
\mathbf{x}\right)  .
\]
In this vein, for any real $p\geq1$, let $\left\vert \mathbf{x}\right\vert
_{p}$ stand for the $l^{p}$ norm of a vector $\mathbf{x}$, and
define\footnote{The parameter $\lambda^{\left(  p\right)  }\left(  G\right)  $
was introduced by Keevash, Lenz and Mubayi in \cite{KLM13} and $\lambda_{\min
}^{\left(  p\right)  }\left(  G\right)  $ was introduced in \cite{Nik14}. In
addition, $\lambda^{\left(  r\right)  }\left(  G\right)  $ is known as the
\emph{spectral radius} of $G$ (see \cite{CoDu11}, \cite{Lim05}, \cite{Qi05},
\cite{Qi13}) and if $r$ is even, then $\lambda_{\min}^{\left(  r\right)
}\left(  G\right)  $ coincides with the \emph{smallest }$H$\emph{-eigenvalue}
of $G$ (see \cite{Qi05}, \cite{Qi13}).}
\[
\lambda^{\left(  p\right)  }\left(  G_{w}\right)  :=\max_{\left\vert
\mathbf{x}\right\vert _{p}=1}\text{ }P_{G_{w}}\left(  \mathbf{x}\right)
\]
and
\[
\lambda_{\min}^{\left(  p\right)  }\left(  G_{w}\right)  :=\min_{\left\vert
\mathbf{x}\right\vert _{p}=1}\text{ }P_{G_{w}}\left(  \mathbf{x}\right)  .
\]

Simple examples show that $\lambda^{\left(  p\right)  }\left(  G_{w}\right)
>0$ and $\lambda_{\min}^{\left(  p\right)  }\left(  G_{w}\right)  <0$, unless
$E\left(  G\right)  =\varnothing$ or $w\left(  e\right)  =0$ for all $e\in
E\left(  G\right)  $; to avoid trivialities these two cases are excluded hereafter.

Observe also that the inequality%
\begin{equation}
P_{G_{w}}\left(  \mathbf{x}\right)  \geq\lambda_{\min}\left(  G_{w}\right)
\left\vert \mathbf{x}\right\vert _{p}^{r} \label{min}%
\end{equation}
holds for every $\mathbf{x}\in\mathbb{R}^{n}$.\medskip

An $r$-graph $G\ $is called $k$\emph{-partite}\textbf{ }if $V\left(  G\right)
$ can be partitioned into $k$ sets so that each edge contains at most one
vertex from each partition set. Likewise, $G\ $is called $k$\emph{-chromatic}%
\textbf{ }if $V\left(  G\right)  $ can be partitioned into $k$ sets so that no
partition set contains an edge.\medskip

Writing $K_{n}^{r}$ for the complete $r$-graph of order $n$, inequality
(\ref{min}) is extended as follows:

\begin{theorem}
\label{th1}Let $r\geq2$ be an even integer, let $k\geq r$ be an integer and
$p\geq r$ be a real number.\ If $G_{w}$ is a weighted $k$-partite $r$-graph,
then%
\begin{equation}
\frac{\lambda^{\left(  p\right)  }\left(  G_{w}\right)  }{\lambda_{\min
}^{\left(  p\right)  }\left(  G_{w}\right)  }\geq\frac{\lambda^{\left(
p\right)  }\left(  K_{k}^{r}\right)  }{\lambda_{\min}^{\left(  p\right)
}\left(  K_{k}^{r}\right)  }.\label{rhin}%
\end{equation}
If $G$ is a complete regular $k$-partite $r$-graph, then equality holds in
(\ref{rhin}).
\end{theorem}

Taking $r=p=2$ and $w\left(  e\right)  =1$ for every $e\in E\left(  G\right)
,$ and noting that $\lambda^{\left(  2\right)  }\left(  K_{k}^{2}\right)
=k-1$ and $\lambda_{\min}^{\left(  2\right)  }\left(  K_{k}^{2}\right)  =-1,$
we obtain Hoffman's inequality (\ref{hin}), so Theorem \ref{th1} indeed
generalizes that result. In fact, the introduction of edge weights extends a
matrix version of (\ref{min}) due to Lov\'{a}sz \cite{Lov79}.

The right side of (\ref{rhin}) needs some clarification. On the one hand, it
is known that
\[
\lambda\left(  K_{k}^{r}\right)  =\left(  k-1\right)  \cdots\left(
k-r+1\right)  k^{r/p},
\]
see, e.g., the proof of Theorem \ref{th1} below. However, the value of
$\lambda_{\min}^{\left(  p\right)  }\left(  K_{k}^{r}\right)  $ is not known
precisely yet. In a forthcoming paper we shall determine the order of
magnitude of $\lambda_{\min}^{\left(  p\right)  }\left(  K_{n}^{r}\right)  $.

One can ask if in Theorem \ref{th1} the premise $k$\emph{-partite} can be
replaced by $k$\emph{-chromatic}$.$ Alas, such change is not straightforward,
if possible at all. Indeed, in Section \ref{cex}, we describe an infinite
family of $2$-chromatic $4$-graphs such that the ratio $\lambda^{\left(
p\right)  }\left(  G\right)  /\left\vert \lambda_{\min}^{\left(  p\right)
}\left(  G\right)  \right\vert $\ grows with the order of the graph, and thus
cannot be bounded by a function of the chromatic number of $G$ and
$p$.\medskip

We proceed with the proof of Theorem \ref{th1}, which is based on an idea of
\cite{Nik07}; a similar idea has been used also in \cite{Ken14}.\medskip

\begin{proof}
[\textbf{Proof of Theorem \ref{th1}}]Suppose that $G_{w}$ is of order $n,$ and
let $V_{1},\ldots,V_{k}$ be a partition of $V\left(  G_{w}\right)  $ such that
any edge of $G$ contains at most one vertex from each $V_{1},\ldots,V_{k}$;
for any $v\in V\left(  G_{w}\right)  $, write $\eta\left(  v\right)  $ for the
unique number such that $v\in V_{\eta\left(  v\right)  }$.

Select $\mathbf{x}\in\mathbb{R}^{n},$ $\mathbf{y}\in\mathbb{R}^{k}$ with
$\left\vert \mathbf{x}\right\vert _{p}=\left\vert \mathbf{y}\right\vert
_{p}=1$ and such that $P_{G_{w}}\left(  \mathbf{x}\right)  =\lambda^{\left(
p\right)  }\left(  G_{w}\right)  ,$ $P_{K_{k}^{r}}\left(  \mathbf{y}\right)
=\lambda_{\min}^{\left(  p\right)  }\left(  K_{k}^{r}\right)  $. Write
$\mathbb{P}$ for the set of all permutations $\sigma:$ $\left[  r\right]
\rightarrow\left[  r\right]  $, and for every $\sigma\in\mathbb{P}$, define a
vector $\mathbf{z}_{\sigma}=\left(  z_{1},\ldots,z_{n}\right)  $ by letting%
\[
z_{i}:=x_{i}y_{\sigma\left(  \eta\left(  i\right)  \right)  },\text{
\ \ \ }i=1,\ldots,n\text{. }%
\]
In view of (\ref{min}), we have
\[
P_{G}\left(  \mathbf{z}_{\sigma}\right)  \geq\lambda_{\min}^{\left(  p\right)
}\left(  G\right)  \left\vert \mathbf{z}_{\sigma}\right\vert _{p}^{r}%
=\lambda_{\min}^{\left(  p\right)  }\left(  G\right)  \left(
{\displaystyle\sum\limits_{i\in\left[  n\right]  }}
\left\vert x_{i}\right\vert ^{p}|y_{\sigma\left(  \eta\left(  i\right)
\right)  }|^{p}\right)  ^{r/p}.
\]
Summing this inequality for all $\sigma\in\mathbb{P},$ we get%
\begin{equation}%
{\displaystyle\sum\limits_{\sigma\in\mathbb{P}}}
P_{G}\left(  \mathbf{z}_{\sigma}\right)  \geq\lambda_{\min}\left(  G\right)
{\displaystyle\sum\limits_{\sigma\in\mathbb{P}}}
\left(
{\displaystyle\sum\limits_{i\in\left[  n\right]  }}
\left\vert x_{i}\right\vert ^{p}|y_{\sigma\left(  \eta\left(  i\right)
\right)  }|^{p}\right)  ^{r/p}\text{ }. \label{main}%
\end{equation}
Our main goals are to calculate $%
{\textstyle\sum\limits_{\sigma\in\mathbb{P}}}
P_{G}\left(  \mathbf{z}_{\sigma}\right)  $ and to bound $%
{\textstyle\sum\limits_{\sigma\in\mathbb{P}}}
$ $\left(
{\textstyle\sum\limits_{i\in\left[  n\right]  }}
\left\vert x_{i}\right\vert ^{p}|y_{\sigma\left(  \eta\left(  i\right)
\right)  }|^{p}\right)  ^{r/p}$ from above$.$

First, by definition,%
\begin{align*}%
{\displaystyle\sum\limits_{\sigma\in\mathbb{P}}}
P_{G}\left(  \mathbf{z}_{\sigma}\right)   &  =r!%
{\displaystyle\sum\limits_{\sigma\in\mathbb{P}}}
\sum_{\left\{  i_{1},\ldots,i_{r}\right\}  \in E\left(  G\right)  }%
w_{i_{1},\ldots,i_{r}}x_{i_{1}}\cdots x_{i_{r}}y_{\sigma\left(  \eta\left(
i_{1}\right)  \right)  }\cdots y_{\sigma\left(  \eta\left(  i_{r}\right)
\right)  }\\
&  =r!\sum_{\left\{  i_{1},\ldots,i_{r}\right\}  \in E\left(  G\right)
}w_{i_{1},\ldots,i_{r}}x_{i_{1}}\cdots x_{i_{r}}%
{\displaystyle\sum\limits_{\sigma\in\mathbb{P}}}
y_{\sigma\left(  \eta\left(  i_{1}\right)  \right)  }\cdots y_{\sigma\left(
\eta\left(  i_{r}\right)  \right)  }.
\end{align*}
Fix an edge $\left\{  i_{1},\ldots,i_{r}\right\}  \in E$ and a permutation
$\sigma\in\mathbb{P}$; the numbers $\sigma\left(  \eta\left(  i_{1}\right)
\right)  ,\ldots,\sigma\left(  \eta\left(  i_{r}\right)  \right)  $ are
distinct and thus determine an edge of $K_{k}^{r},$ say $\left\{  j_{1}%
,\ldots,j_{r}\right\}  .$ There are exactly $r!\left(  k-r\right)  !$
permutations in $\mathbb{P}$ mapping $\left\{  \eta\left(  i_{1}\right)
,\ldots,\eta\left(  i_{r}\right)  \right\}  $ onto $\left\{  j_{1}%
,\ldots,j_{r}\right\}  ,$ implying that
\[%
{\displaystyle\sum\limits_{\sigma\in\mathbb{P}}}
y_{\sigma\left(  \eta\left(  i_{1}\right)  \right)  }\cdots y_{\sigma\left(
\eta\left(  i_{r}\right)  \right)  }=r!\left(  k-r\right)  !\sum_{\left\{
j_{1},\ldots,j_{r}\right\}  \in E\left(  K_{k}^{r}\right)  }y_{j_{1}}\cdots
y_{j_{r}}=\left(  k-r\right)  !\lambda_{\min}^{\left(  p\right)  }\left(
K_{k}^{r}\right)  .
\]
Hence,
\begin{equation}%
{\displaystyle\sum\limits_{\sigma\in\mathbb{P}}}
P_{G}\left(  \mathbf{z}_{\sigma}\right)  =\left(  k-r\right)  !\lambda_{\min
}^{\left(  p\right)  }\left(  K_{k}^{r}\right)  r!\sum_{\left\{  i_{1}%
,\ldots,i_{r}\right\}  \in E}x_{i_{1}}\cdots x_{i_{r}}=\left(  k-r\right)
!\lambda^{\left(  p\right)  }\left(  G_{w}\right)  \lambda_{\min}^{\left(
p\right)  }\left(  K_{k}^{r}\right)  . \label{in1}%
\end{equation}
On the other hand, the Power Mean inequality implies that
\begin{align*}%
{\displaystyle\sum\limits_{\sigma\in\mathbb{P}}}
\text{ }\left(
{\displaystyle\sum\limits_{i\in\left[  n\right]  }}
\left\vert x_{i}\right\vert ^{p}|y_{\sigma\left(  \eta\left(  i\right)
\right)  }|^{p}\right)  ^{r/p}  &  \leq\left(  k!\right)  ^{1-r/p}\text{
}\left(
{\displaystyle\sum\limits_{\sigma\in\mathbb{P}}}
\text{ }%
{\displaystyle\sum\limits_{i\in\left[  n\right]  }}
\left\vert x_{i}\right\vert ^{p}|y_{\sigma\left(  \eta\left(  i\right)
\right)  }|^{p}\right)  ^{r/p}\\
&  =\left(  k!\right)  ^{1-r/p}\left(
{\displaystyle\sum\limits_{i\in\left[  n\right]  }}
\left\vert x_{i}\right\vert ^{p}%
{\displaystyle\sum\limits_{\sigma\in\mathbb{P}}}
|y_{\sigma\left(  \eta\left(  i\right)  \right)  }|^{p}\right)  ^{r/p}.
\end{align*}
Fix an $i\in\left[  n\right]  $, and note that for every $j\in\left[
k\right]  $ there are $\left(  k-1\right)  !$ permutations $\sigma
\in\mathbb{P}$ such that $j=\sigma\left(  \eta\left(  i\right)  \right)  .$
Hence,%
\[%
{\displaystyle\sum\limits_{\sigma\in\mathbb{P}}}
|y_{\sigma\left(  \eta\left(  i\right)  \right)  }|^{p}=\left(  k-1\right)  !%
{\displaystyle\sum\limits_{j\in\left[  k\right]  }}
|y_{j}|^{p}=\left(  k-1\right)  !,
\]
and therefore,%
\[%
{\displaystyle\sum\limits_{i\in\left[  n\right]  }}
\left\vert x_{i}\right\vert ^{p}%
{\displaystyle\sum\limits_{\sigma\in\mathbb{P}}}
|y_{\sigma\left(  \eta\left(  i\right)  \right)  }|^{p}=\left(  k-1\right)  !%
{\displaystyle\sum\limits_{i\in\left[  n\right]  }}
\left\vert x_{i}\right\vert ^{p}=\left(  k-1\right)  !.
\]
Thus, we get%
\begin{equation}%
{\displaystyle\sum\limits_{\sigma\in\mathbb{P}}}
\text{ }\left(
{\displaystyle\sum\limits_{i\in\left[  n\right]  }}
\left\vert x_{i}\right\vert ^{p}|y_{\sigma\left(  \eta\left(  i\right)
\right)  }|^{p}\right)  ^{r/p}\leq\left(  k!\right)  ^{1-r/p}\left(  \left(
k-1\right)  !\right)  ^{r/p}=k^{1-r/p}\left(  k-1\right)  ! \label{in2}%
\end{equation}
Substituting $%
{\textstyle\sum\limits_{\sigma\in\mathbb{P}}}
P_{G}\left(  \mathbf{z}_{\sigma}\right)  $ and $%
{\textstyle\sum\limits_{\sigma\in\mathbb{P}}}
$ $\left(
{\textstyle\sum\limits_{i\in\left[  n\right]  }}
\left\vert x_{i}\right\vert ^{p}|y_{\sigma\left(  \eta\left(  i\right)
\right)  }|^{p}\right)  ^{r/p}$ from (\ref{in1}) and (\ref{in2}) into
(\ref{main}), we get%
\begin{equation}
\lambda^{\left(  p\right)  }\left(  G\right)  \lambda_{\min}^{\left(
p\right)  }\left(  K_{k}^{r}\right)  \geq k^{1-r/p}\left(  k-1\right)
\cdots\left(  k-r+1\right)  \lambda_{\min}^{\left(  p\right)  }\left(
G\right)  . \label{gin}%
\end{equation}

To finish the proof of (\ref{rhin}), we show that
\begin{equation}
\lambda\left(  K_{k}^{r}\right)  =k^{1-r/p}\left(  k-1\right)  \cdots\left(
k-r+1\right)  . \label{lam}%
\end{equation}

Indeed, we can select a nonnegative $\mathbf{u}\in\mathbb{R}^{k}$ with
$\left\vert \mathbf{u}\right\vert _{p}=1$ and $P_{K_{k}^{r}}\left(
\mathbf{u}\right)  =\lambda\left(  K_{k}^{r}\right)  $. Now, Maclaurin's
inequality and the Power Mean inequality imply that
\begin{align*}
\lambda\left(  K_{k}^{r}\right)   &  =P_{K_{k}^{r}}\left(  \mathbf{u}\right)
\leq r!\binom{k}{r}\left(  \frac{u_{1}+\cdots+u_{k}}{k}\right)  ^{r}\leq
k\left(  k-1\right)  \cdots\left(  k-r+1\right)  \left(  \frac{u_{1}%
^{p}+\cdots+u_{k}^{p}}{k}\right)  ^{r/p}\\
&  =k^{1-r/p}\left(  k-1\right)  \cdots\left(  k-r+1\right)  ,
\end{align*}
as claimed.

Clearly, inequality (\ref{rhin}) follows from (\ref{gin}) and (\ref{lam}),
because $\lambda_{\min}^{\left(  p\right)  }\left(  G\right)  <0$ and
$\lambda_{\min}^{\left(  p\right)  }\left(  K_{k}^{r}\right)  <0$.\medskip

It remains to prove that equality holds in (\ref{rhin}) for regular complete
$k$-partite $r$-graphs. Let $G$ be such a graph and let $V_{1},\ldots,V_{k}$
be a partition of $V\left(  G\right)  $ such that any edge of $G$ contains at
most one vertex from each $V_{1},\ldots,V_{k}$. The regularity of $G$ implies
that $\left\vert V_{1}\right\vert =\cdots=\left\vert V_{k}\right\vert ,$ and
so $t=\left\vert V\left(  G\right)  \right\vert /k$ is an integer. As above,
for any $v\in V\left(  G\right)  $, write $\eta\left(  v\right)  $ for the
unique number such that $v\in V_{\eta\left(  v\right)  }$.

Our proof is straightforward: we show that
\begin{equation}
\lambda_{\min}^{\left(  p\right)  }\left(  G\right)  =t^{r-r/p}\lambda_{\min
}^{\left(  p\right)  }\left(  K_{k}^{r}\right)  \label{eq1}%
\end{equation}
and%
\begin{equation}
\lambda^{\left(  p\right)  }\left(  G\right)  =t^{r-r/p}\lambda^{\left(
p\right)  }\left(  K_{k}^{r}\right)  \label{eq2}%
\end{equation}
and these two equations imply that equality holds in (\ref{rhin}).

Select $\mathbf{y}\in\mathbb{R}^{k}$ with $\left\vert \mathbf{y}\right\vert
_{p}=1$ and $P_{K_{k}^{r}}\left(  \mathbf{y}\right)  =\lambda_{\min}^{\left(
p\right)  }\left(  K_{k}^{r}\right)  $ and define a vector $\mathbf{x}%
:=\left(  x_{1},\ldots,x_{n}\right)  $ by letting%
\[
x_{i}=y_{\eta\left(  i\right)  }t^{-1/p},\text{ \ \ \ }i=1,\ldots,n.
\]
On the one hand, we see that
\[%
{\displaystyle\sum\limits_{i\in\left[  n\right]  }}
\left\vert x_{i}\right\vert ^{p}=%
{\displaystyle\sum\limits_{i\in\left[  k\right]  }}
\text{ }%
{\displaystyle\sum\limits_{j\in\left[  V_{i}\right]  }}
\left\vert x_{j}\right\vert ^{p}=%
{\displaystyle\sum\limits_{i\in\left[  k\right]  }}
t\left\vert y_{i}t^{-1/p}\right\vert ^{p}=%
{\displaystyle\sum\limits_{i\in\left[  k\right]  }}
\left\vert y_{i}\right\vert ^{p}=1.
\]
On the other hand,%
\[
P_{G}\left(  \mathbf{x}\right)  =r!t^{r}%
{\displaystyle\sum\limits_{\left\{  j_{1},\ldots,j_{r}\right\}  \in E\left(
K_{k}^{r}\right)  }}
x_{j_{1}}\cdots x_{j_{r}}=r!t^{r-r/p}%
{\displaystyle\sum\limits_{\left\{  j_{1},\ldots,j_{r}\right\}  \in E\left(
K_{k}^{r}\right)  }}
y_{j_{1}}\cdots y_{j_{r}}=t^{r-r/p}\lambda_{\min}^{\left(  p\right)  }\left(
K_{k}^{r}\right)  .
\]
Thus,%
\begin{equation}
\lambda_{\min}^{\left(  p\right)  }\left(  G\right)  =P_{G}\left(
\mathbf{x}\right)  \leq t^{r-r/p}\lambda_{\min}^{\left(  p\right)  }\left(
K_{k}^{r}\right)  . \label{in3}%
\end{equation}

Now, select $\mathbf{x}\in\mathbb{R}^{n}$ with $\left\vert \mathbf{x}%
\right\vert _{p}=1$ and $P_{G}\left(  \mathbf{x}\right)  =\lambda_{\min
}^{\left(  p\right)  }\left(  G\right)  $, and define a vector $\mathbf{y}%
:=\left(  y_{1},\ldots,y_{k}\right)  $ by letting%
\[
y_{i}=t^{1/p-1}%
{\displaystyle\sum\limits_{j\in\left[  V_{i}\right]  }}
x_{j},\text{ \ \ \ }i=1,\ldots,n.
\]
The Power Mean inequality implies that%
\[
\left\vert \mathbf{y}\right\vert _{p}^{p}=%
{\displaystyle\sum\limits_{i\in\left[  k\right]  }}
\left\vert y_{i}\right\vert ^{p}\leq%
{\displaystyle\sum\limits_{i\in\left[  k\right]  }}
\left(
{\displaystyle\sum\limits_{j\in\left[  V_{i}\right]  }}
\left\vert x_{j}\right\vert t^{1/p-1}\right)  ^{p}\leq t^{1-p}%
{\displaystyle\sum\limits_{j\in\left[  V_{i}\right]  }}
t^{p-1}\left\vert x_{j}\right\vert ^{p}=1.
\]
On the other hand,
\[
\lambda_{\min}^{\left(  p\right)  }\left(  G\right)  =P_{G}\left(
\mathbf{x}\right)  =t^{r-r/p}P_{K_{k}^{r}}\left(  \mathbf{y}\right)  \geq
t^{r-r/p}\lambda_{\min}^{\left(  p\right)  }\left(  K_{k}^{r}\right)
\left\vert \mathbf{y}\right\vert _{p}^{r}\geq t^{r-r/p}\lambda_{\min}^{\left(
p\right)  }\left(  K_{k}^{r}\right)  .
\]
In view of (\ref{in3}), equation (\ref{eq1}) is proved.

The proof of (\ref{eq2}) goes along the same lines and is therefore omitted.
\end{proof}

\section{\label{cex}$2$-chromatic $4$-graphs with unbounded $\lambda\left(
G\right)  /\lambda_{\min}\left(  G\right)  $}

Let $n$ be an even integer and define a $4$-graph $G$ as follows: $G$ is of
order $2n$ and its vertex set is the union of two disjoint sets $A$ and $B$,
each of cardinality $n.$ The edges of $G$ consists of all $4$-sets
intersecting $A$ in exactly two vertices. Clearly, $G$ is $2$-chromatic.

\begin{proposition}
If $p\geq2,$ then
\begin{equation}
\lambda^{\left(  p\right)  }\left(  G\right)  =4!\binom{n}{2}^{2}\frac
{1}{\left(  2n\right)  ^{4/p}} \label{sr}%
\end{equation}
and
\begin{equation}
\lambda_{\min}^{\left(  p\right)  }\left(  G\right)  =\Omega\left(
n^{3-4p}\right)  . \label{lev}%
\end{equation}

\end{proposition}

\begin{proof}
For convenience, let $A=\left[  n\right]  ,$ $B=\left[  2n\right]
\backslash\left[  n\right]  .$ Suppose that $\mathbf{x}=\left(  x_{1}%
,\ldots,x_{2n}\right)  $ is a vector with $\left\vert \mathbf{x}\right\vert
_{p}=1$. Set
\[
S_{2}\left(  A\right)  :=%
{\displaystyle\sum\limits_{1\leq i<j\leq n}}
x_{i}x_{j},\text{ \ \ }S_{2}\left(  B\right)  :=%
{\displaystyle\sum\limits_{n<i<j\leq2n}}
x_{i}x_{j},
\]
Clearly, $P_{G}\left(  \mathbf{x}\right)  =4!S_{2}\left(  A\right)
S_{2}\left(  B\right)  $.

Now, Maclaurin's inequality implies that%
\[
S_{2}\left(  A\right)  \leq\binom{n}{2}\left(  \frac{1}{n}%
{\displaystyle\sum\limits_{i\in A}}
\left\vert x_{i}\right\vert \right)  ^{2}\text{ \ and \ \ \ }S_{2}\left(
B\right)  \leq\binom{n}{2}\left(  \frac{1}{n}%
{\displaystyle\sum\limits_{i\in B}}
\left\vert x_{i}\right\vert \right)  ^{2},
\]
and further, the Power Mean inequality implies that
\begin{align*}
P_{G}\left(  \mathbf{x}\right)   &  \leq4!\binom{n}{2}^{2}\left(  \frac{1}{n}%
{\displaystyle\sum\limits_{i\in A}}
\left\vert x_{i}\right\vert \right)  ^{2}\left(  \frac{1}{n}%
{\displaystyle\sum\limits_{i\in B}}
\left\vert x_{i}\right\vert \right)  ^{2}\leq4!\binom{n}{2}^{2}\left(
\frac{1}{n}%
{\displaystyle\sum\limits_{i\in A}}
\left\vert x_{i}\right\vert ^{p}\right)  ^{2/p}\left(  \frac{1}{n}%
{\displaystyle\sum\limits_{i\in B}}
\left\vert x_{i}\right\vert ^{p}\right)  ^{2/p}\\
&  \leq4!\binom{n}{2}^{2}\frac{1}{\left(  2n\right)  ^{4/p}},
\end{align*}
yielding
\[
\lambda^{\left(  p\right)  }\left(  G\right)  \leq4!\binom{n}{2}^{2}\frac
{1}{\left(  2n\right)  ^{4/p}}.
\]
On the other hand, letting $\mathbf{y}:=\left(  y_{1},\ldots,y_{2n}\right)  $
be with $y_{1}=\cdots=y_{2n}=\left(  2n\right)  ^{-1/p},$ we see that
$\left\vert \mathbf{y}\right\vert _{p}=1$ and
\[
\lambda^{\left(  p\right)  }\left(  G\right)  \geq P_{G}\left(  \mathbf{y}%
\right)  =4!\binom{n}{2}^{2}\frac{1}{\left(  2n\right)  ^{4/p}}.
\]
This completes the proof of equation (\ref{sr}).

To prove (\ref{lev}), assume by symmetry that $S_{2}\left(  A\right)  <0$ and
$S_{2}\left(  B\right)  >0$. We see that
\[
2S_{2}\left(  A\right)  =\left(
{\displaystyle\sum\limits_{i\in A}}
x_{i}\right)  ^{2}-%
{\displaystyle\sum\limits_{i\in A}}
x_{i}^{2}>-%
{\displaystyle\sum_{i\in A\cup B}}
x_{i}^{2}>-2n\left(  \frac{1}{2n}%
{\displaystyle\sum_{i\in A\cup B}}
\left\vert x_{i}\right\vert ^{p}\right)  ^{2/p}=-\frac{2n}{\left(  2n\right)
^{2/p}},
\]
and that
\[
2S_{2}\left(  B\right)  =\left(
{\displaystyle\sum\limits_{i\in B}}
x_{i}\right)  ^{2}-%
{\displaystyle\sum\limits_{i\in B}}
x_{i}^{2}<\left(
{\displaystyle\sum_{i\in A\cup B}}
\left\vert x_{i}\right\vert \right)  ^{2}<4n^{2}\left(  \frac{1}{2n}%
{\displaystyle\sum_{i\in A\cup B}}
\left\vert x_{i}\right\vert ^{p}\right)  ^{2/p}=\frac{4n^{2}}{\left(
2n\right)  ^{2/p}}%
\]
Therefore,%
\[
P_{G}\left(  \mathbf{x}\right)  >-4!\frac{n}{\left(  2n\right)  ^{2/p}}%
\cdot\frac{2n^{2}}{\left(  2n\right)  ^{2/p}}=-\frac{4!2n^{3}}{\left(
2n\right)  ^{4/p}},
\]
and so,
\[
\lambda_{\min}^{\left(  p\right)  }\left(  G\right)  >-\frac{4!2n^{3}}{\left(
2n\right)  ^{4/p}}=O\left(  n^{3-4/p}\right)  .
\]
To prove that $\lambda_{\min}^{\left(  p\right)  }\left(  G\right)
=\Omega\left(  n^{3-4/p}\right)  $, recall that $n$ is even and define a
vector $\mathbf{y}:=\left(  y_{1},\ldots,y_{2n}\right)  $ by letting
\[
y_{i}=\left\{
\begin{array}
[c]{ll}%
-\left(  2n\right)  ^{-1/p} & \text{if }1\leq i\leq n/2\text{;}\\
\left(  2n\right)  ^{-1/p} & \text{if }n/2<i\leq2n\text{.}%
\end{array}
\right.
\]
Clearly $\left\vert \mathbf{y}\right\vert _{p}=1,$ and we find that
\[
\lambda_{\min}^{\left(  p\right)  }\left(  G\right)  \leq P_{G}\left(
\mathbf{y}\right)  =-4!\cdot\frac{1}{2}%
{\displaystyle\sum_{i\in\left[  n\right]  }}
y_{i}^{2}%
{\displaystyle\sum\limits_{n<i<j\leq2n}}
y_{i}y_{j}=-\frac{4!n}{2\left(  2n\right)  ^{2/p}}\cdot\binom{n}{2}\frac
{1}{\left(  2n\right)  ^{2/p}}=\Omega\left(  n^{3-4/p}\right)  .
\]
Hence, (\ref{lev}) holds as well, completing the proof of Proposition
\ref{pro1}.
\end{proof}

We see that the graphs $G$ satisfy
\[
\frac{\lambda^{\left(  p\right)  }\left(  G\right)  }{|\lambda_{\min}^{\left(
p\right)  }\left(  G\right)  |}=\Omega\left(  n\right)  ,
\]
so the ratio $\lambda^{\left(  p\right)  }\left(  G\right)  /|\lambda_{\min
}^{\left(  p\right)  }\left(  G\right)  |$ cannot be bounded from above by a
function of the chromatic number of $G$, which is always $2.$

\section{\textbf{Concluding remarks and open problems}}

One can see that Theorem \ref{th1} trivially holds for odd $r$ as well, since
$\lambda_{\min}^{\left(  p\right)  }\left(  G\right)  =-\lambda^{\left(
p\right)  }\left(  G\right)  $ for any $r$-graph $G$ if $r$ is odd. At the
same time we could not find any conspicuous obstacle to meaningful extensions
of the theorem for odd $r$ and for $1\leq p<r$.

Thus, we finish with two open problems:

\begin{problem}
How can Theorem \ref{th1} be meaningfully extended for odd $r$?
\end{problem}

\begin{problem}
Does Theorem \ref{th1} hold for $1\leq p<r$ and even $r$?
\end{problem}

It should be emphasized that in the above problems only tight bounds on
$\lambda^{\left(  p\right)  }\left(  G\right)  /\lambda_{\min}^{\left(
p\right)  }\left(  G\right)  $ would really matter, as it is not too difficult
to prove loose bounds depending solely on the partition number of $G$.

Finally, it would be interesting to determine for which graphs $G$ equality
holds in (\ref{rhin}) in the spirit of Lemma 9.6.2 of \cite{GoRo01}.

\end{document}